\documentclass[11 pt]{amsart}
\usepackage{comment}
\usepackage{enumerate}
\usepackage{amssymb, amsmath}

\def\Z{{\mathbb Z}}

\def\SL{{\rm SL}}
\def\GL{{\rm GL}}

\def\R{{\mathbb R}}

\def\F{{\mathbb F}}

\def\Q{{\mathbb Q}}

\title[Locally soluble Thue equations that are globally insoluble]{A positive proportion of locally soluble quartic Thue equations are globally insoluble}

 \author{Shabnam Akhtari}
 \address{Fenton Hall\\
University of Oregon\\
Eugene, OR 97403-1222 USA}

 \email {akhtari@uoregon.edu}


\begin{document}

\let\realverbatim=\verbatim
\let\realendverbatim=\endverbatim
\renewcommand\verbatim{\par\addvspace{6pt plus 2pt minus 1pt}\realverbatim}
\renewcommand\endverbatim{\realendverbatim\addvspace{6pt plus 2pt minus 1pt}}
\makeatletter
\newcommand\verbsize{\@setfontsize\verbsize{10}\@xiipt}
\renewcommand\verbatim@font{\verbsize\normalfont\ttfamily}
\makeatother

 \newtheorem{thm}{Theorem}[section]
\newtheorem{prop}[thm]{Proposition}
\newtheorem{lemma}[thm]{Lemma}
\newtheorem{cor}[thm]{Corollary}
\newtheorem{conj}[thm]{Conjecture} 

\begin{abstract}
For any fixed nonzero integer $h$, we show that a positive proportion of integral binary quartic forms $F$  do locally everywhere represent $h$, but do not  globally represent $h$.
 We order classes of  integral binary quartic forms  by  the two  generators of their ring of  $\GL_{2}(\Z)$-invariants, classically  denoted by $I$ and $J$.
\end{abstract}

\maketitle


\section{Introduction}\label{Intro}

Let  $h\in\Z$ be nonzero. We will prove the existence of many integral  quartic forms that do not represent  $h$. Specifically, we aim to show many quartic {\it Thue equations} 
\begin{equation}
F(x,y)=h
\end{equation}
have no solutions in  integers $x$ and $y$, where $F(x , y)$ is an irreducible  binary quartic form 
with coefficients in the integers.

Let
$$
F(x , y) = a_{0}x^{4} + a_{1}x^{3}y + a_{2}x^{2}y^{2} + a_{3}xy^{3} + a_{4}y^{4} \in \mathbb{Z}[x , y]. 
$$
The discriminant $D$ of $F(x, y)$ is given by
$$
D = D_{F} = a_{0}^{6} (\alpha_{1} - \alpha_{2})^{2}  (\alpha_{1} - \alpha_{3})^{2}   (\alpha_{1} - \alpha_{4})^{2}   (\alpha_{2} - \alpha_{3})^{2}  (\alpha_{2} - \alpha_{4})^{2}  (\alpha_{3} - \alpha_{4})^{2} ,
$$
where $\alpha_{1}$, $\alpha_{2}$, $\alpha_{3}$ and $\alpha_{4}$ are the roots of 
$$
F(x , 1) =  a_{0}x^{4} + a_{1}x^{3} + a_{2}x^{2} + a_{3}x + a_{4} . 
$$
Let
$
A = \bigl( \begin{smallmatrix}
a & b \\
c & d \end{smallmatrix} \bigr)$ be a $2 \times 2$ matrix, with $a, b, c, d \in \Z$. We define the integral  binary quartic  form $F^{A}(x , y)$ by
$$
F^{A}(x , y) : = F(ax + by ,\  cx + dy).
$$
It follows that
\begin{equation}\label{St6}
D_{F^{A}} = (\textrm{det} A)^{12} D_F.
\end{equation}
If $A \in \GL_{2}(\mathbb{Z})$,  then we say that $\pm F^{A}$ is {\it equivalent} to $F$.

 The $\GL_{2}(\Z)$-invariants of a generic  binary quartic form, which will be called \emph{invariants},  form a ring that is generated by two invariants. These two invariants are denoted by  $I$ and $J$ and are  algebraically independent. For 
 $F(x , y) = a_{0}x^{4} + a_{1}x^{3}y + a_{2}x^{2}y^{2} + a_{3}xy^{3} + a_{4}y^{4}$, these invariants are defined as follows:
\begin{equation}\label{defofI}
I = I_{F} = a_{2}^{2} - 3a_{1}a_{3} + 12a_{0}a_{4} 
\end{equation}
and
\begin{equation}\label{defofJ}
J = J_{F} = 2a_{2}^{3} - 9a_{1}a_{2}a_{3} + 27 a_{1}^{2}a_{4} - 72 a_{0}a_{2}a_{4} + 27a_{0}a_{3}^{2}.
\end{equation}
Every invariant is a polynomial in $I$ and $J$. Indeed,  the discriminant $D$, which is an invariant, satisfies
$$
27D = 4I^3 - J^2.
$$

Following \cite{BaShSel}, we define the height $\mathcal{H}(F)$ of an integral binary  quartic form $F(x , y)$ as follows,
\begin{equation}\label{Bash}
\mathcal{H}(F) : = \mathcal{H}(I , J) := \max\left\{\left|I^3\right|, \frac{J^2}{4}\right\},
\end{equation}
where $I = I_F$ and $J = J_F$.

We note that if  $F(x,y)=h$ has no solution, and $G$ is a {\it proper subform} of $F$, i.e., 
\begin{equation}\label{defofsubform}
G(x,y)=F(ax+by,cx+dy)
\end{equation}
for some integer matrix $A=\bigl(\begin{smallmatrix}a&b\\c&d\end{smallmatrix}\bigr)$ with $|\!\det A|>1$, then clearly $G(x,y)=h$ will also have no integer solutions.  We will call a binary form {\it maximal} if it is not a proper subform of another binary form.

Our goal  in this paper is to show  that many (indeed, a positive proportion) of integral binary quartic forms   are not proper subforms, locally represent $h$ at every place, but globally do not represent~$h$. The following is our main result.
\begin{thm}\label{mainquartic}
Let $h$ be any nonzero integer. When maximal integral binary quartic forms $F(x , y) \in \mathbb{Z}[x , y]$ are ordered by their height $\mathcal{H}(I, J)$, a positive proportion of the $\GL_2(\Z)$-classes of these forms $F$ have the following  properties:
\begin{enumerate}[{\rm (i)}]
\item they locally everywhere represent $h$ $($i.e., $F(x , y) = h$ has a solution in~$\R^2$ and in~$\Z_p^2$ for all $p);$ and
\item  they globally do not represent $h$ $($i.e., $F(x , y) = h$ has no solution in~$\mathbb{Z}^2)$. 
\end{enumerate}
\end{thm}

In other words, we show that a positive proportion of quartic Thue equations $F(x,y)=h$ fail the integral Hasse principle, when classes of 
integral binary quartic forms $F$ are ordered by the height $\mathcal{H}(I , J)$ defined in \eqref{Bash}. We will construct a family of quartic forms that do not represent a given integer $h$  and obtain a
 lower bound $\mu > 0$ for the density of such forms.  The value for $\mu$ is expressed explicitly in \eqref{finaldensity}.  Moreover, our method yields an explicit construction of this positive density of forms.
It is conjectured that, for any $n \geq 3$, a density of  $100\%$ of integral binary  forms of degree $n$ that locally represent a fixed integer $h$ do not globally represent $h$. The positive lower bound $\mu$ in \eqref{finaldensity} is  much smaller than  the conjectured  density $1$.

In  joint work with Manjul Bhargava  \cite{AB}, we proved  a  result similar to Theorem \ref{mainquartic}. In \cite{AB} we consider integral binary forms of any given degree ordered by na\"ive height (the maximum of absolute values of their coefficients).  Theorem \ref{mainquartic} is new, as we use a different ordering of integral binary quartic forms, which is more interesting  for at least two reasons;  here integral binary quartic forms are ordered by two quantities $I$ and $J$, as opposed to five coefficients, and  $I$ and $J$, unlike the coefficients,  are $\GL_{2}(\Z)$-invariant.
In \cite{AB},  for any fixed integer $h$, we showed that a positive proportion of binary forms of degree $n \geq 3$ do not represent $h$, when binary $n$-ic forms are ordered by their naive heights. Moreover, for $n =3$, we established the same conclusion when cubic forms are ordered by their absolute  discriminants. The Davenport-Heilbronn Theorem, which states that the number of equivalence classes of irreducible binary cubic forms per discriminant is a constant on average, was an essential part of our argument in \cite{AB} for cubic forms. More importantly we made  crucial use of the asymptotic counts given by the Davenport-Heilbronn Theorem   for the number of equivalent integral cubic forms with bounded absolute discriminant (see the original work in \cite{DH}, and  \cite{AB} for application and further references). Such results are not available for  binary forms of degree larger than $3$. For quartic forms, fortunately we are empowered by beautiful results due to Bhargava and Shankar that  give asymptotic formulas for the number of $\GL_{2}(\Z)$-equivalence classes of irreducible integral binary quartic forms having bounded invariants. These results will be discussed in Section \ref{BaShsec}.

This paper is organized as follows.  In Section \ref{perilim} we  discuss some  upper bounds for the number of primitive solutions of quartic Thue equations. Section \ref{BaShsec} contains important results, all cited from \cite{BaShSel}, about the height $\mathcal{H}(I, J)$.   
In Sections \ref{splitsection} and \ref{localsection} we impose conditions on the splitting behavior of the forms  used  in our construction modulo different primes to make sure we produce a large enough number of forms (which in fact form a subset of  integral quartic forms with positive density) that do not represent $h$, without any local obstruction. 
 In Section \ref{completesection}, we summarize the assumptions made in  Sections \ref{splitsection} and \ref{localsection}, and apply essential results cited in Sections  \ref{perilim} and \ref{BaShsec} to conclude that the quartic forms that we construct  form a subset of integral binary quartic forms with positive density.

\section{Primitive Solutions of Thue Equations}\label{perilim}

Let $F(x , y) \in \Z[x , y]$ and $m \in \Z$.
A pair $(x_{0} , y_{0}) \in \mathbb{Z}^2$ is called a {\it primitive solution} to the Thue equation $F(x , y) = m$ if  $F(x_{0} , y_{0}) = m$ and  
$\gcd(x_{0} , y_{0}) = 1$.
We will use the following result from \cite{AkhQuaterly} to obtain upper bounds for the number of primitive solutions of Thue equations.

\begin{prop}[\cite{AkhQuaterly}, Theorem 1.1]\label{maineq4}
 Let $F(x , y) \in \mathbb{Z}[x , y]$ be an irreducible binary form of degree $4$ and discriminant $D$.  Let $m$ be an integer with
  $$
0 < m  \leq \frac{|D|^{\frac{1}{6} - \epsilon} }  {(3.5)^{2} 4^{ \frac{2}{3 } } },
 $$
 where $ 0< \epsilon < \frac{1}{6}$.
 Then the equation $|F(x , y)| = m$ has at most
  \[ 
  36 + \frac{4}{3 \epsilon}
    \]
   primitive solutions. In addition to the above assumptions, if we assume that the polynomial $F(X , 1)$ has $2 \mathtt{i}$ non-real roots, with $\mathtt{i} \in\{0, 1, 2\}$, then the number of primitive solutions does not exceed
  \[ 
 36 -16\mathtt{i} + \frac{4-\mathtt{i}}{3 \epsilon}.
   \]
  \end{prop}

  If the integral binary forms $F_{1}$ and $F_{2}$ are equivalent, as defined in the introduction, then  there exists $A \in \GL_2(\Z)$ such that
  $$
  F_2(x , y) = F_1^{A}(x , y) \, \, \textrm{or} \, \, F_2(x , y) = -F_1^{A}(x , y).
  $$
  Therefore,
  $D_{F_{1}} = D_{F_{2}}$, and for every fixed integer $h$, the number of primitive solutions to $F_1(x , y) = \pm h$ equals the number of primitive solutions to $F_2(x , y) = \pm h$.
  
The invariants $I_F$ and $J_F$  of an integral quartic form $F$ that are defined in \eqref{defofI} and \eqref{defofJ} have weights $4$ and $6$, respectively. This means 
\begin{equation}\label{Idet}
I_{F^{A}} = (\textrm{det} A)^{4} I_F,
\end{equation}
and 
\begin{equation}\label{Jdet}
J_{F^{A}} = (\textrm{det} A)^{6} J_F.
\end{equation}
Consequently, by definition of the height $\mathcal{H}$ in \eqref{Bash}, we have
\begin{equation}\label{Hdet}
\mathcal{H}(F^{A}) = (\textrm{det} A)^{12} \mathcal{H}(F),
\end{equation}
and
\begin{equation*}
\mathcal{H}(-F^{A}) = (\textrm{det} A)^{12} \mathcal{H}(F).
\end{equation*}

\section{On the Bhargava--Shankar height $\mathcal{H}(I, J)$}\label{BaShsec}

In \cite{BaShSel} Bhargava and Shankar introduce the height $\mathcal{H}(F)$ (see \eqref{Bash} for definition) for any integral binary quartic form $F$.  In this section we present some of the asymptotical  results in \cite{BaShSel}, which will be used in our proofs. Indeed these asymptotic formulations  are the reason that we are able to order quartic forms with respect to their $I$ and $J$ invariants.

One may ask which integer  pairs $(I , J)$ can actually occur as the invariants of an integral binary quartic form.  The following result of Bhargava and Shankar provides a complete answer to this question.
\begin{thm}[\cite{BaShSel}, Theorem 1.7]\label{BaSh-thm1.7}
A pair $(I , J) \in \mathbb{Z} \times \mathbb{Z}$ occurs as the invariants of an integral binary quartic form if and only if it satisfies one of the following congruence conditions:
\begin{eqnarray*}
(a) \, \,  I \equiv 0 \, \, (\textrm{mod}\, \,  3) &\textrm{and}\, & J  \equiv  0\, \,  (\textrm{mod}\, \, 27),\\
(b)\,  \,  I \equiv 1 \, \, (\textrm{mod}\, \,  9) &\textrm{and}\, & J  \equiv  \pm 2\, \,  (\textrm{mod}\, \, 27),\\
(c)\, \,  \,  I \equiv 4 \, \, (\textrm{mod}\, \,  9) &\textrm{and}\, & J  \equiv  \pm 16\, \,  (\textrm{mod}\, \, 27),\\
(d)\,  \,  I \equiv 7 \, \, (\textrm{mod}\, \,  9) &\textrm{and}\, & J  \equiv  \pm 7\, \,  (\textrm{mod}\, \, 27).
\end{eqnarray*}
\end{thm}

Let $V_{\R}$ denote the vector space of binary quartic forms over the real numbers $\R$. The  group $\GL_{2}(\R)$ naturally acts on $V_{\R}$. The action of $\GL_{2}(\Z)$  on $V_{\R}$  preserves the lattice $V_{\Z}$ consisting of the integral elements of $V_{\R}$.
The elements of  $V_{\Z}$ are the forms that we are interested in.
Let $V^{(\mathtt{i})}_{\Z}$ denote the set of elements in $V_{\Z}$ having nonzero discriminant 
and $\mathtt{i}$ pairs of complex conjugate roots and $4 -2\mathtt{i}$  real roots.

For any $\GL_{2}(\Z)$-invariant set $S \subseteq V_{\Z}$, let $N(S ; X)$ denote the number of $\GL_{2}(\Z)$-equivalence classes of irreducible elements $f \in S$ satisfying 
$\mathcal{H}(f) < X$. 

For any set $S$ in  $V_{\Z}$ that is definable by congruence conditions, following \cite{BaShSel}, we denote by $\mu_{p}(S)$ the $p$-adic density of the $p$-adic closure of $S$ in  $V_{\Z_p}$, where we normalize the additive measure $\mu_p$ on  $V_{\Z_p}$ so that  $\mu_p(V_{\Z_p})= 1$. The following is  a combination of Theorem 2.11 and  Theorem 2.21 of \cite{BaShSel}.

\begin{thm}[Bhargava--Shankar]\label{BaSh-thm2.11}
Suppose $S$ is a subset of $V_{\Z}$ defined by congruence conditions modulo finitely many prime powers, or even a suitable infinite set of prime powers. Then we have
\begin{equation}
N(S \cap  V_{\Z}^{(\mathtt{i})}; X) \sim N( V_{\Z}^{(\mathtt{i})}; X) \prod_{p} \mu_{p} (S).
\end{equation}
\end{thm}

The  statement of Theorem \ref{BaSh-thm2.11}  for finite number of congruence conditions  follows directly from Theorem 2.11 of \cite{BaShSel}.  
In Subsection 2.7 of \cite{BaShSel},
 some congruence conditions are specified that are suitable for inclusion of  infinitely many primes in the statement of Theorem \ref{BaSh-thm2.11} (see Theorem 2.21 of \cite{BaShSel}).

A function $\phi : V_{\Z} \rightarrow [0, 1]$ is said to be \emph{defined by congruence conditions} if, for all primes $p$, there exist functions 
$\phi_p : V_{\Z_p} \rightarrow [0, 1]$
satisfying the following conditions:\newline
 (1) for all $F \in V_{\Z}$, the product $\prod_{p} \phi_{p}(F)$  converges to $\phi(F)$,\newline
 (2) for each prime $p$,  the function $\phi_p$ is locally constant outside some closed
set $S_p \subset V_{\Z_p}$  of  measure zero.
Such a function $\phi$ is called \emph{acceptable} if, for sufficiently large primes $p$, we have $\phi_p(F) = 1$ whenever $p^2 \nmid D_F$.

For our purpose,  particularly in order to impose congruence conditions modulo the infinitely many primes that are discussed  in Subsection \ref{largeprimesubsection}, we define the acceptable function $\phi : V_{\Z} \rightarrow \{0, 1\}$ to be the characteristic  function of a certain subset  of integral binary quartic forms. More specifically, for $p < 49$, we define $\phi_p$ to be the constant function $1$. For $p > 49$, we define $\phi_p : V_{\Z_p} \rightarrow \{0, 1\}$ to be  the characteristic  function of the set of integral binary quartic forms that are not  factored as $c_p M_p (x , y)^2$ modulo $p$, with  $c_p \in \mathbb{F}_p$ and $M_p(x , y)$ any quadratic form over $\mathbb{F}_p$. Then 
\begin{equation}\label{defofaccept}
\phi(F) = \prod_{p} \phi_{p}(F)
\end{equation} is the characteristic function of the set of integral binary quartic forms that are not  factored as $c_p M_p (x , y)^2$ over $\mathbb{F}_p$ for any $p > 49$.
We denote by $\lambda(p)$ the $p$-adic density 
$\int_{F \in V_{\Z_p}} \phi_p(F) dF
$.
 The value of $\lambda(p)$ will be computed in \eqref{largedensity}. It turns out that in Theorem \ref{mainquartic}, the positive proportion of integral  binary quartic  forms that do not represent $h$  is  bounded below by 
 $$
\mu = \kappa(h) \prod_{p} \lambda(p),
$$
 where $p$ ranges over all primes and $\kappa(h)$ is a constant that only depends on $h$ and can be explicitly determined 
 from   \eqref{finaldensity} in Section 6.

Later in our proofs, in order to construct many inequivalent  quartic forms, it will be important to work with quartic forms that have no non-trivial stabilizer in $\GL_2(\mathbb{Z})$.
We note that the stabilizer in $\GL_2(\mathbb{Z})$ of an element in $V_{\mathbb{R}}$  always contains the identity matrix and its negative, and has size at least $2$.
We will appeal to another  important  result due to Bhargava and Shankar, which bounds the number of $\GL_{2}(\mathbb{Z})$-equivalence classes of integral binary quartic forms having large stabilizers inside 
$\GL_{2}(\mathbb{Z})$. 
\begin{prop}[\cite{BaShSel}, Lemma 2.4]\label{BSL2.4}
 The number of $\textrm{GL}_{2}(\mathbb{Z})$-orbits of integral binary quartic forms 
 $F \in V_{\mathbb{Z}}$ such that $D_F \neq 0$  and $\mathcal{H}(F) < X$ whose stabilizer in $\GL_{2}(\mathbb{Q})$ has size greater than $2$  is $O(X^{3/4 + \epsilon})$.
 \end{prop}

\section{Quartic Forms Splitting Modulo a Prime}\label{splitsection}

\textbf{Definition}. We define the subset $V'_{\mathbb{Z}}$ of integral binary quartic forms $V_{\mathbb{Z}}$ to be those forms $F$ that have trivial stabilizer (of size $2$).

By Proposition \ref{BSL2.4},  $V'_{\mathbb{Z}}$ is a dense subset of equivalence  classes of quartic forms  and selecting our forms from $V'_{\mathbb{Z}}$ will not alter the $p$-adic densities that we will present later. From now on we will work only with classes of  forms in $V'_{\mathbb{Z}}$.

\textbf{Definition}.  Assume that $F(x , y)$ is an irreducible quartic form.  We say that $F(x , y)$ \emph{splits completely} modulo a prime number $p$, if  either
\begin{equation}\label{splitgI}
 F(x , y) \equiv m_{0} (x - b_{1}y)(x-b_{2}y) (x-b_{3}y)(x- b_{4}y)\, \, (\textrm{mod} \, \,   p),
 \end{equation}
 or
 \begin{equation}\label{splitgII}
 F(x , y) \equiv m_{0} y(x-b_{2}y) (x-b_{3}y)(x- b_{4}y)\, \, (\textrm{mod} \, \,   p),
 \end{equation}
  where $m_{0} \not \equiv 0$ (mod $p$), and $b_{1}, b_{2}, b_{3}, b_{4}$ are distinct integers modulo $p$, and further
  \begin{equation}\label{assumemore}
    b_{2}, b_{3}, b_{4} \not \equiv 0  \, \, \qquad  (\textrm{mod} \, \, p).
    \end{equation}
   In case \eqref{splitgI}, we
  call $b_1$, $b_2$, $b_3$, and $b_4$ the \emph{simple roots} of the binary form $F(x , y)$ modulo $p$.   In case \eqref{splitgII}, we
  call $\infty$, $b_2$, $b_3$, and $b_4$ the \emph{simple roots} of the binary form $F(x , y)$ modulo $p$.

    Let $p \geq 5$ be a prime. The $p$-adic density of binary quartic forms that split completely modulo $p$ is given by
 \begin{eqnarray}\label{splitdensity}
    \mu_{p} &= & \frac{ (p -1) \left( \frac{p (p-1)(p-2) (p-3) }{4!}  + \frac{(p-1)(p-2)(p-3)} {3!} \right)  }{p^5}\\ \nonumber
   & =& \frac{ (p -1)^2 (p+4) (p-2) (p-3)     }{4! \, p^5},
    \end{eqnarray}
where in the first identity in \eqref{splitdensity}, the summand $\frac{p (p-1)(p-2) (p-3) }{4!}$ in the numerator counts the corresponding forms in \eqref{splitgI} and  the summand 
$\frac{(p-1)(p-2)(p-3)} {3!}$ counts the corresponding forms in \eqref{splitgII}.
Clearly the factor $p -1$ in the numerator counts the number of possibilities for $m_{0}$ modulo $p$ and  the denominator $p^5$ counts all quartic forms with all choices for their five coefficients modulo $p$.

Now assume $F(x , y)$ is an irreducible integral quartic form that splits completely modulo $p$.
 For $j\in \{ 1, 2, 3, 4\}$, we define
 \begin{equation}\label{defofFb}
 F_{b_{j}}(x , y)  : =  F(p x + b_{j} y, y),
 \end{equation}
 and additionally   in case \eqref{splitgII},
 \begin{equation}\label{defofFinf}
F_{\infty}(x , y) := F (p y , x).
\end{equation}

We claim that 
the four forms  $F_{b_{1}}(x , y)$ (or $F_{\infty}(x,y)$),  $F_{b_{2}}(x , y)$,  $F_{b_{3}}(x , y)$, and  $F_{b_{4}}(x , y)$ are pairwise inequivalent.  Indeed, any transformation $B\in\GL_2(\Q)$ taking, say $F_{b_i}(x,y)$ to $F_{b_j}(x,y)$ must be of the form $B=\bigl(\begin{smallmatrix}p&b_i\\ 0& 1\end{smallmatrix}\bigr)^{-1}\!A\bigl(\begin{smallmatrix}p &b_j\\ 0& 1\end{smallmatrix}\bigr)$, where $A\in\GL_2(\Q)$ stabilizes $F(x,y)$.  
Since we assumed $F \in V'_{\mathbb{Z}}$, the $2 \times 2$ matrix $A$ must be the identity matrix or its negative, and so $B= \pm \bigl(\begin{smallmatrix}p&b_i\\ 0& 1\end{smallmatrix}\bigr)^{-1}\bigl(\begin{smallmatrix}p&b_j\\ 0& 1\end{smallmatrix}\bigr)$.  But $B\notin\GL_2(\Z)$, as $p \nmid (b_i-b_j)$. Therefore, for $i \neq j$, 
the quartic forms $F_{b_i}(x,y)$ and $F_{b_j}(x,y)$  are not $\GL_2(\Z)$-equivalent.

Similarly in case \eqref{splitgII}, any  transformation $B\in\GL_2(\Q)$ taking $F_{\infty}(x,y)$ to $F_{b_j}(x,y)$ must be of the form $B= \bigl(\begin{smallmatrix}0&p\\ 1& 0\end{smallmatrix}\bigr)^{-1}\!A\bigl(\begin{smallmatrix}p &b_j\\ 0& 1\end{smallmatrix}\bigr)$, where $A\in\GL_2(\Q)$ stabilizes $F(x,y)$.  
This change-of-variable matrix does not belong to $\GL_{2}(\mathbb{Z})$, unless $b_{j}\equiv 0$ (mod $p$).
Therefore,  $F_{\infty}(x , y)$, $F_{b_2}(x , y)$, $F_{b_{3}}(x , y)$, and  $F_{b_{4}}(x , y)$ are pairwise inequivalent, as long as none of $b_{2}$, $b_{3}$ and $b_{4}$ are a multiple of $p$ (this motivated the extra assumption \eqref{assumemore} in our definition).
 Starting with  a form $F$ that belongs to $V'_{\mathbb{Z}}$ and splits completely modulo $p$, we can construct $4$ integral quartic forms that are pairwise inequivalent.

Let 
$
F(x , y) = a_{0}x^4 + a_{1} x^{3} y +a_{2} x^{2} y^2 + a_{3} x y^3+ a_{4}y^4 \in \mathbb{Z}[x , y],
$ 
with content $1$ (i.e., the integers $a_{0}, a_{1},  a_{2}, a_{3},  a_{4}$ have no common prime divisor).
If $F(x , y)$ satisfies \eqref{splitgII} then
\begin{equation}\label{deftildeinf}
\tilde{F}_{\infty}(x , y):= \frac{ F_{\infty}(x , y)}{p} \in \mathbb{Z}[x , y],
\end{equation}
where $F_{\infty}(x , y)$ is defined in \eqref{defofFinf}.
Suppose that 
\begin{equation}\label{alessp4}
F (b , 1) \equiv 0 \, \, \, (\textrm{mod}\, \, p), \,  \, \, \textrm{with}\, \, b \in \mathbb{Z}.
\end{equation}
 By \eqref{defofFb},
\begin{equation*}\label{Faei4}
F_{b}(x , y) =    F(p x + by , y) =  e_{0} x^4 + e_{1} x^{3} y +e_{2} x^{2} y^2 + e_{3} x y^3+ e_{4}y^4,
\end{equation*}
with
\begin{equation}\label{dotss4}
 e_{4-j} = p^j \sum_{i=0}^{4-j} a_{i} \, b^{4-i-j} {4-i \choose j},
 \end{equation}
 for $j=0, 1, 2, 3, 4$.
  If $j \geq 1$, clearly $e_{4-j}$ is divisible by $p$.  Since  $e_{4} = F(b , 1)$, by \eqref{alessp4}, $e_{4}$ is also divisible by $p$. Therefore,
  \begin{equation}\label{deftildeb}
\tilde{F_{b}}(x , y): = \frac{F_{b}(x , y)}{p} \in  \mathbb{Z}[x , y].
\end{equation}
 Since $e_{3} = p f'(b)$, where $f'(X)$ denotes the derivative of polynomial $f(X) = F(X , 1)$, if $b$ is a simple root modulo $p$  then $f'(b)\not \equiv 0\, \, (\textrm{mod}\, p)$ and  
\begin{equation}\label{yL4}
\tilde{F_{b}}(x , y) = y^{3} L(x , y)\, \,  (\textrm{mod}\, \, p),
\end{equation}
 where $L(x , y)= l_{1}x + l_{2}y$  is a linear  form modulo $p$, with $l_{1} \not \equiv 0\pmod p$.  
 
 We also note that $\mathcal{H}(F_b)$, defined in \eqref{Bash}, as well as  the invariants of the form $F_b$,  can be expressed in terms of invariants of the form $F$, as $F_b$  is obtained under the action of a $2 \times 2$ matrix of determinant $\pm p$ on $F$. By \eqref{Idet}, \eqref{Jdet}, and \eqref{Hdet}, we have
  \begin{eqnarray*}\label{Hoftilde}
D_{F_{b}} & =& p^{12} D_F,\\
I_{{F_{b}}} &= & p^4 I_{{F}}, \\
 J_{{F_{b}}} &= & p^6 J_{F}, \\
 \mathcal{H}\left({F_{b}}\right) & =&  \mathcal{H}\left(I_{{F_{b}}}, J_{{F_{b}}} \right) = p^{12}  \mathcal{H}(F).
\end{eqnarray*}
After multiplication of the form $F_{b}(x , y)$ by $p^{-1}$, we therefore have
\begin{eqnarray}\nonumber
D_{\tilde{F_{b}}} & =& p^{6} D_F\\ \nonumber
I_{\tilde{F_{b}}} &= &  p^2 I_{{F}}, \\ \nonumber
 J_{\tilde{F_{b}}} &= &  p^3 J_{F}, \\
 \mathcal{H}\left(\tilde{F_{b}}\right) & =&  \mathcal{H}\left(I_{\tilde{F_{b}}}, J_{\tilde{F_{b}}} \right) = p^{6}  \mathcal{H}(F).
\end{eqnarray}

Now let us consider the quartic Thue equation
$$
F(x , y) = m,
$$
where $m =  p_{1} p_{2} p_{3} h$, and  $p_{1}$, $p_{2}$, and $p_{3}$ are three distinct primes greater than $4$, and $\gcd(h, p_{k}) = 1$, for $k\in \{ 1, 2, 3\}$. We will further assume that the quartic form $F(x , y)$ splits completely modulo $p_{1}$, $p_{2}$, and $p_{3}$.  In Lemma \ref{corresponds-sol}, we will construct $64$ integral binary quartic forms $G_{j}(x , y)$, for $1 \leq j \leq  4^3$, and will  make a one-to-one correspondence between the set of primitive solutions of $F(x , y) = m$ and the union of the sets of  primitive solutions  of $G_{j}(x , y) = h$, for $1 \leq j \leq  4^3$.  First we need two auxiliary lemmas.

\begin{lemma}\label{lem1corres}
Let $F(x , y) \in \mathbb{Z}[x , y]$ be a binary quartic form that splits completely modulo $p$ and $m = p m_1$, with  $p \nmid m_1$. The primitive solutions of the Thue equation $F(x , y) = m$ are in one-to-one correspondence with the union of the sets of primitive solutions to four  Thue equations
$$
\tilde{F}_{i}(x , y) = m_1,
$$
where $\tilde{F}_{i}(x , y)$ are defined in \eqref{deftildeinf} and  \eqref{deftildeb}, and $i=1, 2, 3, 4$.
\end{lemma}
\begin{proof}
Assume that $(x_{0}, y_{0}) \in \mathbb{Z}^2$ is a solution to $F(x , y) =  m = p m_{1}$.  
If
$$
 F(x , y) \equiv m_{0} (x - b_{1}y)(x-b_{2}y) (x-b_{3}y)(x- b_{4}y)\, \, (\textrm{mod} \, \,   p),
 $$
then
since  
$
p| F(x_0 , y_0)
$,
we have
$$
p| (x_{0}- b_{i} y_0)
$$
for some $i \in \{1, 2, 3, 4\}$. The value of $i$ is uniquely determined by the solution $(x_{0}, y_{0})$, as $b_{j}$'s are distinct modulo $p$. Therefore, 
\begin{equation}\label{x0X}
x_{0} = p_1 X_{0} + b_{i} y_{0},
\end{equation}
for some $ X_{0} \in \mathbb{Z}$, and $(X_{0}, y_{0})$ is a solution to 
\begin{equation}\label{redmp}
\tilde{F}_{i}(x , y)  = \frac{1}{p} F(p x + b_{i} y , y) = m_{1} = \frac{m}{p}.
\end{equation}
Conversely, assume for a fixed $i \in \{1, 2, 3, 4\}$ that  $(X_{0}, y_{0}) \in \mathbb{Z}^2$ is a solution to 
$$
\tilde{F}_{i}(x , y)  = \frac{1}{p} F(p x + b_{i} y , y) = m_{1} = \frac{m}{p}.
$$
First we observe that $p \nmid y_{0}$. Because otherwise $p$ divides  $p X_0 + b_{i} y_0$ and $p^4 \mid \frac{m}{p}$, which is a contradiction.
Now by construction of the form $\tilde{F}_{i}(x , y)$, we clearly have $(x_0 , y_{0})$, with
$$
x_{0} = p X_{0} + b_{i} y_{0},
$$
 satisfies  the equation $F(x , y) = m$. Further, if $(X_{0}, y_{0})$ is a primitive solution of 
 $\tilde{F}_{i}(x , y)  = \frac{m}{p}$, since $p \nmid y_0$, we have $\gcd(x_0 , y_0) = 1$. 
 
 Assume that
 $$
 F(x , y) \equiv m_{0} y (x-b_{2}y) (x-b_{3}y)(x- b_{4}y)\, \, (\textrm{mod} \, \,   p).
 $$
The pair $(x_{0} , y_{0}) \in \mathbb{Z}^2$ with $p \nmid y_{0}$ is a primitive solution 
 of 
 $$
F(x , y) = p m_1,
$$
if and only if $p \mid (x_0-b_{2}y_0) (x_0-b_{3}y_0)(x_0- b_{4}y_0)$. In this case, for a unique $i \in \{2, 3, 4\}$, we have \eqref{x0X}, and  $(X_0, y_0)$ is a primitive solution to the Thue equation \eqref{redmp}.
Similarly, the pair $(x_{1} , y_{1}) \in \mathbb{Z}^2$ with $p \mid y_{1}$ is a primitive solution 
 of 
 $$
F(x , y) = p m_1,
$$
if and only if $(Y_1, x_{1})$, with $Y_1 = \frac{y_1}{p}$, is a  primitive solution to 
$$\tilde{F}_{\infty}(x , y) = \frac{m}{p}.$$
\end{proof}

\begin{lemma}\label{lem2}
If  $F(x, y)$ splits completely modulo $p_1$ and $p_2$, then $\tilde{F}_{b}(x , y)$ will also split completely modulo $p_2$,  for any simple root $b$ (possibly $\infty$) of $F(x , y)$ modulo $p_1$.
\end{lemma}
\begin{proof}
If
$$
F(x , y) \equiv m_{0} (x - b_1 y) (x - b_2 y) (x - b_3 y) (x - b_4 y) \,  \qquad (\textrm{mod} \, \, p_{1})
$$
and 
\begin{equation}\label{ciroots}
F(x , y) \equiv m'_{0} (x - c_1 y) (x - c_2 y) (x - c_3 y) (x - c_4 y) \,  \qquad (\textrm{mod} \, \, p_{2}),
\end{equation}
then for any $b \in \{ b_{1}, b_{2}, b_{3}, b_{4}\}$, we have
$$
\tilde{F_{b}}(x , y) \equiv  m''_{0}(x - c'_1 y) (x - c'_2 y) (x - c'_3 y) (x - c'_4 y) \,  \qquad (\textrm{mod} \, \, p_{2}),
$$
where 
$$
c'_{j} = p_{1} c_{j} + b.
$$
The integers  $c'_1, c'_2 , c'_3 , c'_4$ are   indeed distinct  modulo  $p_{2}$, as  $c_1, c_2 , c_3 , c_4$ are so and $p_{1}$ is invertible modulo $p_{2}$. We conclude that the quartic form $ \tilde{F}_{b}(x , y)$  splits completely modulo $p_{2}$, as well.

If 
 $$
F(x , y) \equiv m_{0} y (x - b_2 y) (x - b_3 y) (x - b_4 y) \,  \qquad (\textrm{mod} \, \, p_{1})
$$
and \eqref{ciroots} holds, then 
$$
\tilde{F}_{\infty}(x , y) \equiv  m''_{0}(x - c'_1 y) (x - c'_2 y) (x - c'_3 y) (x - c'_4 y) \,  \qquad (\textrm{mod} \, \, p_{2}),
$$
with $c'_i = c^{-1}_{i}$ modulo $p_2$, where $0$ and $\infty$ are considered to be  the inverse of each other modulo $p_2$. Namely, if $c_1 =0$ modulo $p_2$, we get 
$$
\tilde{F}_{\infty}(x , y) \equiv  m''_{0} y (x - c'_2 y) (x - c'_3 y) (x - c'_4 y) \,  \qquad (\textrm{mod} \, \, p_{2}).
$$

If
$$
F(x , y) \equiv m_{0} y (x - b_2 y) (x - b_3 y) (x - b_4 y) \,  \qquad (\textrm{mod} \, \, p_{1})
$$
and 
\begin{equation*}
F(x , y) \equiv m'_{0} y (x - c_2 y) (x - c_3 y) (x - c_4 y) \,  \qquad (\textrm{mod} \, \, p_{2}),
\end{equation*}
then
$$
\tilde{F}_{\infty}(x , y) \equiv  m''_{0}x  (x - c'_2 y) (x - c'_3 y) (x - c'_4 y) \,  \qquad (\textrm{mod} \, \, p_{2}),
$$
 with $c'_i = c^{-1}_{i}$ modulo $p_2$. Therefore, if $F(x, y)$ splits completely modulo $p_1$ and $p_2$, the $\tilde{F}_{b}(x , y)$ will also split completely modulo $p_2$,  for any simple root $b$  of $F(x , y)$ modulo $p_1$.
\end{proof}

\begin{lemma}\label{corresponds-sol}
Let $h$ be an integer, and
$p_1$, $p_{2}$, and $p_{3}$ be three distinct primes greater than $4$ that do not divide $h$. Let $F(x, y) \in \mathbb{Z}[x , y]$ be a  binary quartic form that splits completely modulo primes $p_{1}$, $p_{2}$, and $p_{3}$. Then there are $64$  binary quartic forms $G_{i}(x , y) \in \mathbb{Z}[x , y]$, with $1 \leq i \leq 64$, such that every primitive solution $(x_{\mathit{l}}, y_{\mathit{l}})$  of the equation $F(x , y)= h \, p_{1} p_{2} p_{3}$ corresponds uniquely to a triple $(j, x_{l, j}, y_{l, j})$, with
$$
j \in \{1, 2, \ldots, 64\},\, \,  x_{\mathit{l}, j}, y_{\mathit{l}, j} \in \mathbb{Z}, \, \, \gcd(x_{\mathit{l}, j} , y_{\mathit{l}, j}) =1,
$$
and
$$
G_{j} (x_{\mathit{l}, j} , y_{\mathit{l}, j}) = h.
$$
Furthermore, 
\begin{equation*}
  \mathcal{H}\left( G_{j} \right)  = \left(p_1 p_2 p_3\right)^{6}  \mathcal{H}(F),
  \end{equation*}
  for $j = 1, \ldots, 64$.
\end{lemma}
\begin{proof}
Let $m = p_1 p_2 p_3 h$. 
By Lemma \ref{lem1corres}, we may
reduce the Thue equation  $F(x , y) = m$ modulo $p_1$ to obtain $4$ quartic Thue equations
 \begin{equation}\label{reduceto4}
 \tilde{F}_{i}(x , y) = \frac{m}{p_1},
  \end{equation} 
 with $i = 1, 2, 3, 4$, such that   every primitive solution of $F(x , y)= h \, p_{1} p_{2} p_{3} = m$ corresponds uniquely to a primitive solution of exactly one of the equations in \eqref{reduceto4}.

 By Lemma \ref{lem2}, every
 binary quartic form $\tilde{F}_{i}(x , y)$ in \eqref{reduceto4} splits completely modulo $p_2$. 
 Applying Lemma \ref{lem1corres} modulo $p_2$  to each equation in \eqref{reduceto4}, we construct $4$   binary quartic  forms. Therefore,  we obtain $4^2$ Thue equations
  \begin{equation}\label{reduceto16}
 \tilde{F}_{i, k}(x , y) = \frac{m}{p_1p_2},
  \end{equation}   
  with $i, k =1, 2, 3, 4$,
   such that  every primitive solution $F(x , y)= h \, p_{1} p_{2} p_{3} = m$ corresponds uniquely to a primitive solution of exactly one of the equations in \eqref{reduceto16}.
   By \eqref{Hoftilde}, 
    \begin{equation}\label{HofFij}
  \mathcal{H}\left( F_{i, k} \right)  = \left(p_1 p_2 \right)^{6}  \mathcal{H}(F).
  \end{equation}

  By Lemma \ref{lem2},
 each form $\tilde{F}_{i, k}(x , y)$ splits modulo   $p_3$. We may apply Lemma \ref{lem1corres} once again to each equation in  \eqref{reduceto16}. This way we obtain $4^3$  equations 
  \begin{equation}\label{reduceto64}
G_{j}(x , y) = \frac{m}{ p_1 p_2 p_3} = h.
\end{equation}
The construction of these equations ensures a one-to-one correspondence between the primitive solutions of the equation $F(x , y) = m$ and the union of the  sets of the primitive   solutions of Thue equations in \eqref{reduceto64}.

 By \eqref{Hoftilde} and \eqref{HofFij},
 \begin{equation}\label{HofGj}
  \mathcal{H}\left( G_{j} \right)  = \left(p_1 p_2 p_3\right)^{6}  \mathcal{H}(F),
  \end{equation}
  for $j = 1, \ldots, 64$.
\end{proof}

We note that if $F(x , y)$ is  irreducible over $\mathbb{Q}$,  its associated forms $G_{j} (x , y)$, which are constructed in the proof of Lemma \ref{corresponds-sol}, will also  be irreducible over $\mathbb{Q}$ as all of the matrix actions are rational.  Furthermore, the forms $G_{j}(x , y)$ are not constructed as proper subforms of the  binary quartic form $F(x , y)$.
 Indeed, they are maximal over~$\Z_p$ for all $p\notin \{p_1,p_{2},p_3\}$ (being equivalent, up to a unit constant, to $F(x,y)$ over~$\Z_p$ in that case), while for $p\in\{p_1,p_2, p_3\}$, we have $p\nmid D_F$, implying $p^6 || D_{G_j}$, and so $G_j(x , y)$ cannot be a subform over $\Z_p$ of any form by equation~(\ref{St6}) (see the definition of a subform in \eqref{defofsubform}).

We remark that the  reduction of Thue equations $F(x , y) = m$ modulo prime divisors of $m$ is a classical approach, and some sophisticated applications of it to bound the number of solutions of Thue equations can be found in \cite{Bom, Ste}. 

\section{Avoiding Local Obstructions}\label{localsection}

In the previous section, we constructed $4^3$ binary quartic forms $G_{j}(x , y)$ and established Lemma \ref{corresponds-sol},  which corresponds each primitive solution of $F(x , y) = h p_1 p_2 p_3$  to a primitive  solution of one of the equations $G_{j}(x , y) = h$, for $1 \leq j \leq  4^3$. 
 Using Proposition \ref{maineq4}, we will obtain a small upper bound for the number of integral solutions to the equation $F(x , y) = m = p_{1} p_{2} p_{3} h$, which will lead us to  conclude that some of the newly constructed Thue equations $G_{j}(x , y)= h$ cannot have any solutions.

 In this section we will work with a proper subset of the set of all quartic forms  to construct forms such that the associated Thue equations have no local obstructions to solubility.
We will impose some extra congruence conditions in our choice of forms $F(x , y)$, resulting in construction of $4^3$ forms $G_i(x , y)$ that locally represent $h$. For each prime $p$, we will make some congruence assumptions modulo $p$ and present  $p$-adic densities for the subset of quartic forms that satisfy our assumptions to demonstrate that we will be left with a subset of $V_{\mathbb{Z}}$ with positive density.

Before we divide up our discussion modulo different primes, we note that by \eqref{defofsubform},
if a form is non-maximal, then either it is not primitive, or after an $\SL_2(\Z)$-transformation it is of the form $a_0x^4+a_1x^{3}y+a_2 x^2 y^2 +a_3 xy^3+a_4 y^4$, where $p^i\mid a_i$, $i=0,1,2, 3, 4$, for some prime~$p$.  In particular, integral  binary quartic forms that are non-maximal must factor modulo some prime $p$ as a constant times the forth power of a linear form.  It turns out that all integral  binary quartic forms that are discussed in this section are indeed maximal.

\subsection{Quartic Forms Modulo $2$.}

To ensure that a quartic Thue equation $F(x , y) =h$ has a solution  in $\mathbb{Z}_2$, it is sufficient to assume that
$$
F(x , y) \equiv L_1(x,y) L_2(x , y)^{3} \, \, (\textrm{mod}\, \,  2^{4}),
$$
where $L_1(x,y)$ and $L_2(x,y)$ are linearly independent linear forms modulo $2$. The system of two linear equations 
$$
 L_1(x,y) \equiv h \, \, (\textrm{mod}\,  2^{4})
 $$
 $$
  L_2(x , y) \equiv 1 \, \, (\textrm{mod}\,  2^{4}),
  $$
  has a solution and therefore, by Hensel's Lemma, $F(x , y) = h$ is soluble in $\mathbb{Z}_2$.

The 2-adic density of quartic forms $F(x , y)$ such  that $
F(x , y) \equiv L_1(x,y) L_2(x , y)^{3}$ modulo $2^{4}$ is
\begin{equation}\label{2-adicdensity}
\frac{6}{2^5} = \frac{3}{16},
\end{equation}
where the linear forms $L_1$ and $L_2$ can be chosen from the three linear forms $x$, $y$, or $x + y$.

It is indeed necessary to consider integral quartic forms modulo $16$, as a $2$-adic unit $u$ belongs to $\mathbb{Q}^{4}_{2}$ if and only if $u \equiv 1$ modulo $16 \mathbb{Z}_2$. 
More specifically,  assume that $(x_0: y_0: z_0)$ is a $\mathbb{Z}_2$-point on the projective curve $C: hz^4 = F(x , y)$ and
$u = z_0^4$, with $z_0$ a unit in $\mathbb{Z}_2$. Therefore, $z_0 = 1 +2 t$ for some $t \in \mathbb{Z}_2$ and
$$
z_{0}^{4} = (1 + 2 t)^4  \equiv 1 + 8 \left(t(3t+1)\right) \equiv 1 \, \, (\textrm{mod}\, \, 16).
$$

\subsection{Quartic Forms Modulo Large Primes}\label{largeprimesubsection}

Let us consider the curve $C: h z^4 = F(x , y)$ of genus~$g = 3$ over the finite field $\mathbb{F}_{q}$ of order $q$. By the Leep-Yeomans generalization of Hasse-Weil bound  in \cite{LeYe},  the number of points $N$ on the curve $C$ satisfies the inequality
\begin{equation}\label{HWLeYe}
\left| N - (q+1) \right| \leq 2g \sqrt{q}.
\end{equation}

 Let $p$ be a prime $p> (2g+1)^2 = 49$, $p \not \in \{p_1, p_2, p_3\}$, $p \nmid h$. Since $p+1 \geq 2g \sqrt{p} +1$, the lower bound in \eqref{HWLeYe} is nontrivial, implying that there must be an $\mathbb{F}_p$-rational point on the curve $h z^4 = F(x , y)$.

 If there exists   $a \in \Z$ such that
\begin{equation}\label{SimpleRoot}
F(x, y) \equiv (x- ay) A(x , y)\, \,  (\textrm{mod}\,  p),
\end{equation}
with $A(x , y)$ an integral cubic binary form for which
\begin{equation}\label{SimpleRoota}
A(a , 1) \not \equiv 0 \, \, (\textrm{mod}\, p),
\end{equation}
then 
by Hensel's lemma,
the smooth $\mathbb{F}_p$-point  $(x_0:  y_0 : z_0) = (a : 1 : 0)$  will lift to a 
$\mathbb{Z}_p$-point on the curve $h z^4 = F(x , y)$.
Similarly, if 
\begin{equation*}
F(x, y) \equiv y A(x , y)\, \,  (\textrm{mod}\,  p),
\end{equation*}
with $A(x , y)$ an integral cubic binary form for which
\begin{equation*}
A(1 , 0) \not \equiv 0 \, \, (\textrm{mod}\, p),
\end{equation*}
the smooth $\mathbb{F}_p$-point  $(x_0: y_0 : z_0) = (1 : 0 : 0)$  will lift to a 
$\mathbb{Z}_p$-point on the curve $h z^4 = F(x , y)$.

A quartic form  over $\F_p$  that has a triple root  must have a simple root, as well. 
So we will assume
that $F(x,y)$ does not factor as $cM(x,y)^2$ modulo~$p$ for any quadratic binary form $M(x,y)$ and constant $c$ over $\F_p$. By definition, these forms are  maximal over $\Z_p$. It follows from this assumption on $F(x ,y)$ that the curves $hz^4=F(x,y)$ are irreducible over $\F_p$  and there is at least one smooth $\mathbb{F}_p$-rational point on $h z^4 = F(x , y)$, which lifts to a $\mathbb{Z}_p$-point.

We conclude that the integral quartic  forms $G_{j}(x, y)$, constructed as described in Section \ref{splitsection} from such a form $F(x , y)$, all represent $h$ in $\Z_p$ for primes $p> (2g+1)^2$ as well. 

The $p$-adic density of binary quartic  forms  that are primitive and not constant multiples 
of  the second  powers of quadratic binary forms modulo~$p$ is
\begin{equation}\label{largedensity}
 1 - \frac{ (p-1)(p+1) p }{ 2p^5} -  \frac{ (p-1)(p+1)  }{p^5},
\end{equation}
where  the summand $-\frac{ (p-1)(p+1) p }{ 2p^5}$ eliminates forms  of the shape $ c M^2(x, y) = c (x-b_{1}y)^2 (x-b_{2}y)^2 $ or $ c M^2(x, y) = c (x-b_{1}y)^2 y^2$ (mod $p$), and the summand  $- \frac{ (p-1)(p+1)  }{p^5}$ eliminates  forms of the shape  $ c L(x , y)^4$ (mod $p$), with $L(x , y)$ a linear form modulo $p$.

\subsection{Quartic Forms Modulo Special Odd Primes}\label{specialoddprime}

For $p \mid h$ we will assume that 
$$
F(x , y) \equiv L_1(x,y) L_3(x , y)^{3}\, \, \textrm{ (mod}\, \,  p),
$$
where $L_1(x,y)$ and $L_2(x,y)$ are two linearly independent linear forms modulo $p$.
To find $\Z_p$-points on the curve   $C: hz^4 = F(x , y)$, we consider the equation $F(x , y)=0$ (mod $p$).
Since $L_1(x , y)$ and $L_2(x , y)$ are linearly independent modulo $p$, the system of linear equations
$$
 L_1(x,y) \equiv 0 \textrm{ (mod} \, \, p)
$$
 and 
 $$
 L_2(x,y) \equiv 0 \textrm{ (mod} \, \, p)
 $$
has exactly one solution.
Since $L_1(x , y)=0$ has at least three points over $\F_p$, the equation $F(x , y)=0$ (mod $p$) has at least two solutions over $\F_p$  that provide smooth $\F_p$-points on the curve  $C: hz^4 = F(x , y)$ (i.e., all points other than that intersection point of the two lines defined by $L_1(x , y)$ and $L_2(x , y)$).  By Hensel's Lemma, these smooth points will lift to $\Z_p$-points. Thus the equations $F(x,y)=h$ 
and $G_j(x,y)=h$ will  be  locally soluble modulo $p$.

Similarly, for every odd prime $p  \not \in \{ p_{1}, p_2, p_{3}\}$, with $ p <  49$  and $p \nmid h$ (these are the primes  not considered in Subsection \ref{largeprimesubsection}),  we will assume that 
$$
F(x , y) \equiv L_1(x,y) L_2(x , y)^{3}\, \, \textrm{ (mod}\, \,  p),
$$
where $L_1(x,y)$ and $L_2(x,y)$ are linear forms that are linearly independent modulo~$p$.
This condition implies that $F(x , y) \equiv h \textrm{ (mod} \, \, p)$ has solutions in integers, for $L_1(x,y)$ and $L_2(x,y)$ are linearly independent  and therefore we can find $x_{0} , y_{0} \in \mathbb{Z}$ satisfying the following system of linear equations:
 $$
 L_1(x_0,y_0) \equiv h \textrm{ (mod} \, \, p)
$$
 and 
 $$
 L_2(x_0,y_0) \equiv 1 \textrm{ (mod} \, \, p).
 $$
The smooth $\F_p$-point $(x_0 : y_0 : 1)$ lifts to a  $\Z_p$-point on the curve $C: hz^4 = F(x , y)$.

The $p$-adic density of primitive binary quartic forms of the shape
\begin{equation}\label{modhigherp}
 L_1(x,y) L_2(x , y)^{3}\, \, (\textrm{mod} \, \,  p)
\end{equation}
where $L_1(x,y)$ and $L_2(x,y)$ are linearly independent linear forms modulo $p$ and 
 is 
\begin{equation}\label{specialdensity}
\frac{(p+1)p(p-1)}{p^{5}}.
\end{equation}
The above density is calculated by considering the unique  factorization of the form $F$ modulo $p$ as
$$
m_{0} (x - b_{1}y)(x - b_{2} y)^3,
$$
with $m_{0}$ non-zero, and $b_{1}$ and $b_{2}$ distinct roots (possibly $\infty$) modulo $p$.
Such forms are maximal over $\Z_p$

\section{Completing the proof}\label{completesection}

For $i=1, 2, 3$, let $p_{i}$ be the $i$-th prime greater than   $4$ such that $p_{i}\nmid h$ and set
 $$
 m = h\, p_1 p_2 p_3,
$$ 
and 
$$
\mathcal{P} = \{ p_1,  p_2,  p_3\}.
$$
 For example, if $h=1$, we will choose $p_1 = 5$, $p_2 = 7$, and $p_3 = 11$.
  Let $F(x , y)$ be a  maximal primitive irreducible integral  binary quartic form  which has a trivial stabilizer  in $\GL_{2}(\mathbb{Q})$, with   
 $$
\left|D_F   \right| > (3.5)^{24} \,  4^{8} \left( \prod_{i=1}^{3} p_{i} \right)^{12}.
 $$
 We note that the above assumption on the size of the discriminant of quartic forms  exclude only finitely many $\GL_{2}(\mathbb{Z})$-equivalence classes of quartic forms (see \cite{BM, EG1}).

In order to ensure that $h$ is represented by $F$ in $\mathbb{R}$, we assume that the leading coefficient of $F$ is positive if $h$ is positive and negative otherwise.  Assume further that $F(x , y)$  splits completely modulo the primes $p_{1}$, $p_2$, $p_{3}$.

Assume that for every prime $p \not \in \{ p_{1}, p_2, p_3\} = \mathcal{P}$, with $ p < 49$, we have
$$
F(x , y) \equiv L_1(x,y) L_2(x , y)^{3}\, \, \textrm{ (mod}\, \,  p),
$$
where $L_1(x,y)$ and $L_2(x,y)$ are linear forms that are linearly independent modulo~$p$.

Finally, assume, for each prime $p >  49$,
that $F(x,y)$ does not factor as $cM(x,y)^2$ modulo~$p$ for  any quadratic binary form $M(x,y)$ and constant $c$ over $\F_p$.

By Proposition  \ref{maineq4}, and taking $\epsilon = \frac{1}{12}$,  there are at most 
 \[ 
 36 -16\mathtt{i} + \frac{4-\mathtt{i}}{\frac{1}{4}} = 52 - 20 \mathtt{i}
   \]
  primitive  solutions to the equation 
   $$
   F(x , y) = m =  h\, p_1 p_2 p_3,
   $$
   where $2 \, \mathtt{i}$ is the number of non-real roots of the polynomial $F(X, 1)$. 

By Lemma \ref{corresponds-sol}, each primitive solution $(x_{0}, y_{0})$ of $F(x , y) = m$ corresponds uniquely to a solution of $G_{i}(x , y) = h$, where $1 \leq i \leq 4^3$ is also uniquely determined by 
$(x_{0}, y_{0})$. 
Since 
$$
4^3 - 52 + 20 \mathtt{i} = 12 +  20 \mathtt{i} \geq 12
$$
we conclude that at least  $12$ of the $64$ equations $G_{i}(x , y) = h$ have no solutions in integers $x, y$.

By \eqref{HofGj}, and Theorems 
\ref{BaSh-thm1.7} and \ref{BaSh-thm2.11},
 we have the following lower bound $\mu$ for the density of integral quartic forms that represent $h$ locally, but not globally,
\begin{equation}\label{finaldensity}
\mu = \frac{12 }  {\left(p_1 p_2 p_3\right)^{5}} \, \,  \delta_{2} \prod_{p \in \mathcal{P}} \sigma(p)  \prod_{p\geq 49, \,  p\not\in \mathcal{P}, \, p\nmid h} \lambda(p) \prod_{p \mid h \, \textrm{or}\, p < 49} \gamma_{p},  
\end{equation}
where, via \eqref{splitdensity},
\eqref{2-adicdensity},
\eqref{largedensity},
\eqref{specialdensity},
$$\delta_2 =  \frac{3}{16},$$
$$\sigma(p)= \frac{ (p -1)^2 (p+4) (p-2) (p-3)     }{4! \, p^5},$$ 
\begin{equation}\label{lambdacal}
\lambda(p) = 1 - \frac{ (p-1)(p+1) p }{ 2p^5} -  \frac{ (p-1)(p+1)  }{p^5}, 
\end{equation}
 and 
 $$\gamma(p) = \frac{(p+1)p(p-1)}{p^{5}}.$$
  In \eqref{finaldensity} all products are over rational primes. 
 For all but finitely many primes $p$, the density 
 $
 \lambda(p)$
 in \eqref{lambdacal}
  contributes to the  product in \eqref{finaldensity}.  Since 
  $$
  \prod_p \left(1 - \frac{ (p-1)(p+1)^2}{ 2p^5}\right)
  $$
  is a convergent Euler product, the lower bound $\mu$ is a real number satisfying $0 <\mu <1$.


\section*{Acknowledgements.}  I am grateful to the anonymous referee for their careful reading of an earlier version of this manuscript  and insightful comments. I would like to thank Arul Shankar for very helpful conversation and answering my questions, especially regarding the height $\mathcal{H}(I , J)$, which is the key tool in the present paper.
 I would also like to thank Mike Bennett and  Manjul Bhargava  for their insights and suggestions.
  This project was initiated during my visit to the
Max Planck Institute for Mathematics, in Bonn, in the academic year 2018-2019. I acknowledge the support from the MPIM. In different stages of this project,  my research has been partly 
 supported by the National Science Foundation  award DMS-2001281  and by the 
Simons Foundation Collaboration Grants, Award Number 635880.



\end{document}